\documentclass[reqno,12pt]{amsart}

\usepackage{amssymb,amsmath,psfrag}
\usepackage{enumerate}
\usepackage{graphicx}
\usepackage{hyperref}
\usepackage{mathrsfs}

\newcommand\sectionpage{}
\renewcommand\sectionpage{\newpage}

\newtheorem{lem}{Lemma}[section]

\newtheorem{prop}[lem]{Proposition}
\newtheorem{thm}[lem]{Theorem}
\newtheorem{conj}[lem]{Conjecture}
\newtheorem{qn}[lem]{Question}
\theoremstyle{definition}

\newtheorem{exam}[lem]{Example}
\newtheorem{construction}[lem]{Construction}

\numberwithin{equation}{section}
\numberwithin{table}{section}
\numberwithin{figure}{section}

\newcommand\bbF{\mathbb{F}}

\newcommand\bbZ{\mathbb{Z}}

\newcommand\cL{{\mathcal L}}

\newcommand\cO{{\mathcal O}}
\newcommand\cP{{\mathcal P}}

\newcommand\cS{{\mathcal S}}

\renewcommand\ell{l}

\newcommand\bL{\mathbf{L}}
\newcommand\PR{\ensuremath{\mathbb{PR}}}

\newcommand\PP{\Pi}

\renewcommand\setminus\smallsetminus

\begin{document}

%%%%%%%%%%%%%%%%%%%%%%%%%%%%%%%%%%%%%%%%%%%%%%%%%%%%%%%%%%%%%%%%%%%%%%%%%%%%%%%%%%

\thispagestyle{empty}

\title{Projective Rectangles: The Graph of Lines}

\author{Rigoberto Fl\'orez}
\thanks{Fl\'orez's research was partially supported by a grant from The Citadel Foundation.}
\address{Dept.\ of Mathematical Sciences, The Citadel, Charleston, South Carolina 29409}
\email{\tt rigo.florez@citadel.edu}

\author{Thomas Zaslavsky}
\address{Dept.\ of Mathematics and Statistics, Binghamton University, Binghamton, New York 13902-6000}
\email{\tt zaslav@math.binghamton.edu}

\date{\today}

\begin{abstract}
A projective rectangle is like a projective plane that may have different lengths in two directions.  We develop properties of the graph of lines, in which adjacency means having a common point, especially its strong regularity and clique structure.  The main construction of projective rectangles, stated in a previous paper, gives rectangles whose graph of lines is a known strongly regular bilinear forms graph.  That fact leads to a proof that the main construction does produce projective rectangles, and also gives a new representation of bilinear forms graphs.  We conclude by mentioning a few simple graph properties, such as the chromatic number, which is not known, and a partial geometry obtained from the graph.
\end{abstract}

%51E26  Other finite linear geometries
%05B35 Combinatorial aspects of matroids and geometric lattices
%05C22 Signed and weighted graphs
%51A30 Desarguesian and Pappian geometries
%51A45 Incidence structures embeddable into projective geometries
%51A99 None of the above, but in this section (Linear incidence geometry) DUBIOUS
%51E20 Combinatorial structures in finite projective spaces DUBIOUS


\maketitle

\pagestyle{myheadings}
\markboth{Fl\'orez and Zaslavsky}{Projective Rectangles: Graph of Lines}

\setcounter{tocdepth}{3}
\tableofcontents

%============
\sectionpage\section {Introduction}\label{sec:intro}

A projective rectangle, which we introduced in \cite{pr1}, is like a projective plane that may have different sizes in two directions.  Projective rectangles include projective planes as trivial examples, when the two sizes are equal, but otherwise they have properties such as partial Desargues's theorems that are not common to all projective planes.

In this sequel we examine finite projective rectangles through the graph of lines, whose vertices are the short lines with adjacency defined by concurrence in a point.  This graph is strongly regular (Section \ref{sec:graphlines}).  Its clique structure enables us to prove a main theorem in \cite{pr1}, which we were unable to prove by pure incidence geometry, that a ``subplane construction'' in a finite Desarguesian projective plane creates projective rectangles.  The graph of a rectangle resulting from that construction turns out to be a known type of strongly regular graph called a bilinear forms graph, $H_q(2,k)$, where the prime power $q$ is the order of a Desarguesian plane (see Section \ref{sec:srg})---but our construction of the graph appears to be a new one.  

The graph of lines of any finite projective rectangle has the same strong regularity parameters as $H_q(2,k)$ (if we let $q$ be any natural number), but we do not know whether all finite projective rectangles are constructed by the subplane construction nor whether all projective rectangles with the same parameters have isomorphic graphs of lines.  The broadest problem, that of classifying finite projective rectangles, is open.

In Section \ref{sec:properties} we look at simple properties of the graph of lines, such as planarity and chromatic number, finding both answers and questions.
In Section \ref{sec:partialgeometry} we observe that every graph of lines generates a partial geometry using certain maximal cliques called point cliques, but the other maximal cliques, called plane cliques, do not; this is known for $H_q(2,k)$ but we see it is a consequence of the axioms of a projective rectangle.

Why are we studying the graph of lines when the only finite example we know, $H_q(2,k)$, is already well known?  There are two answers.  First, we have a new approach to this graph via the axioms of an incidence geometry.  Second, there may be new examples not isomorphic to any $H_q(2,k)$.  We think the first reason justifies our study, and we hope that in future examples may be discovered that are presently unknown.

%============
\sectionpage
\section {Projective rectangles}\label{sec:pr}

We present essential properties of projective rectangles from \cite{pr1}.

An \emph{incidence structure} is a triple $(\cP,\mathcal{L},\mathcal{I})$ of sets with $\mathcal{I}
\subseteq \cP \times \mathcal{L}$. The elements of $\cP$
are \emph{points}, the elements of $\mathcal{L}$ are \emph{lines}.
A point $p$ and a line $l$ are \emph{incident} if $(p,l) \in \mathcal{I}$. A set $P$ of points is said to be
\emph{collinear} if all points in $P$ are in the same line. We say that two
distinct lines \emph{intersect in a point} if they are incident with the same point.  
If the lines are sets of points, then incidence is containment and we may omit $\mathcal{I}$ from the notation.

A \emph{projective rectangle} is an incidence structure $(\cP,\mathcal{L},\mathcal{I})$ that satisfies the following axioms:

\begin{enumerate} [({A}1)]
\item \label{Axiom:A1}   Every two distinct points are incident with exactly one line.
\medskip

\item \label{Axiom:A2} There exist four points with no three of them collinear.
\medskip

\item \label{Axiom:A3}  Every line is incident with at least three distinct points.
\medskip

\item \label{Axiom:A4}  There is a \emph{special point} $D$.
A line incident with $D$ is called \emph{special}.  A line that is not incident with $D$ is called \emph{ordinary}, and a point that is not $D$ is called \emph{ordinary}.
\medskip

\item \label{Axiom:A5}  Each special line intersects every other line in exactly one point.
\medskip

\item \label{Axiom:A6}  If two ordinary lines $l_1$ and $l_2$ intersect in a point, then every two lines that intersect both  $l_1$ and $l_2$ in four distinct points, intersect in a point.
\end{enumerate}

A projective plane is a projective rectangle.  We call it \emph{trivial}.  A projective rectangle in which there are no two disjoint ordinary lines is a projective plane.  Our interest is in the other projective rectangles.

If a projective rectangle $\PR$ has $m+1$ special lines, each with $n+1$ points, then we say that the
\emph{order} of $\PR$ is $(m,n)$.  In this article we always assume $m$ and $n$ are finite.  

We note elementary properties from \cite[Section 3]{pr1}: $n \geq m\geq 2$ and a projective rectangle with $m=n$ is trivial.  Every special line has the same number of points, every ordinary line has $m+1$ points.   There are exactly $n^2$ ordinary lines; the number of ordinary lines incident with an ordinary point is $n$.  There are exactly $(m+1)n$ ordinary points.
For a given point $p$ in an ordinary line $l$, there are $n-1$ ordinary lines intersecting $l$ at $p$.
The point set of $\PR\setminus D$ is partitioned by the special lines deleting $D$.

\begin{exam}\label{ex:L2k}
The matroid $L_2^k$ (see Figure \ref{figure1}) is a projective rectangle with $m+1=3$ special lines and $n = 2^k$ ordinary points on each.  Let $A:= \left\{ a_g \mid g \in \bbZ_2^k \right\}
\cup \{D \}$, $B:= \left\{ b_g \mid g \in \bbZ_2^k \right\} \cup \{D \}$ and
$C:= \left\{ c_g \mid g \in \bbZ_2^k \right\} \cup \{D \}$.  
Let $L_2^k$ be the simple matroid of rank 3 defined on the ground set
$E:= A\cup B\cup C$ by its rank-2 flats; they are $A$, $B$, $C$, which are the special lines,
and the sets $\{a_g, b_{g +h}, c_h \}$ with $g$ and $h$ in $\bbZ_2^k$, which are the ordinary lines.

We say more about this in Examples \ref{TensorProduct4} and \ref{ex:narrow} and Section \ref{sec:srg}.

\begin{figure} [htbp]
\begin{center}
\includegraphics[width=9cm]{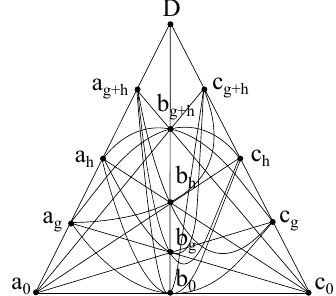}
\caption{The matroid $L_2^2$ with group the Klein 4-group, $\bbZ_2 \times \bbZ_2 \cong \{1,g,h,g+h\}$.  } \label{figure1}
\end{center}
\end{figure}
\end{exam}

A \emph{subplane} of $\PR$ is an incidence substructure that is a projective plane.
A subplane may contain an ordinary line and all its points; such a subplane has order $m$.

\begin{thm}\label{prop:maximalsubplane}
In a projective rectangle, every maximal subplane has order $m$.
\end{thm}

When we refer to a \emph{plane} in a projective rectangle, we mean a maximal subplane.  Also, when we say several lines are \emph{coplanar}, we mean there is a plane $\pi$ such that each of the lines that is ordinary is a line of $\pi$ and for each line $s$ that is special, $s \cap \pi$ is a line of $\pi$.

\begin{prop}\label{prop:Dinpp}
In a projective rectangle $\PR$, the special point $D$ is a point of every plane.  Also, for every special line $s$ and every plane $\pi$, $s\cap\pi$ is a line of $\pi$.
\end{prop}

\begin{thm}[Planes in \PR]\label{prop:twolinesintersectingpp} Let $\PR$ be a projective rectangle. If two ordinary lines in $\PR$ intersect in a point, then both are lines of a unique plane in $\PR$.  

\end{thm}

\begin{thm}\label{thm:counting}
Let $\PR$ be a projective rectangle of order $(m,n)$.  
\begin{enumerate}[{\rm(a)}]

\item \label{prop:counting:plonpi} There are $m(m+1)$ ordinary points and $m^2$ ordinary lines in each plane.   

\item \label{prop:counting:pionl} The number of planes that contain each ordinary line is $\dfrac{n-1}{m-1}.$  

\item \label{prop:counting:pi} The number of planes in $\PR$ is $\dfrac{n^2(n-1)}{m^2(m-1)}.$  

\end{enumerate}
\end{thm}

\begin{thm}[{\cite[Corollary 4.6]{pr1}}]\label{thm:constraint}
A nontrivial projective rectangle $\PR$ of order $(m,n)$ has $n \geq m^2$.
\end{thm}

%======
\sectionpage\section{The graph of lines}\label{sec:graphlines}

A projective rectangle $\PR$ gives rise to a graph, the intersection graph of ordinary lines of $\PR$. We define a graph whose vertices are the ordinary lines and in which 
two lines are adjacent if they intersect; equivalently, by Theorem \ref{prop:twolinesintersectingpp}, they are adjacent if they lie in the same projective plane.  This graph is called the \emph{graph of lines} associated to
$\PR$; it is denoted by $G_\cL(\PR)$. 

The ordinary lines in a plane $\pi$ form a clique in $G_\cL(\PR)$, which we call a \emph{plane clique}.  
A special line in $\PR$ may not be a line in $\pi$; however, its restriction to $\pi$ is.  Therefore, $m+1$ lines in $\pi$ are restrictions of special lines and the other lines in $\pi$ are ordinary lines.  By Theorem \ref{thm:counting} Part \eqref{prop:counting:plonpi} there are $m^2$ ordinary lines in $\pi$.  Thus, a plane clique has order $m^2$.  There is another kind of clique that consists of all the ordinary lines on any one ordinary point, which we call a \emph{point clique}.  This clique has order $n$.  No point clique is a plane clique, but the two have the same order if $n=m^2$, which is the case for the smallest nontrivial projective planes (by Theorem \ref{thm:constraint}).

An $r$-regular graph $G := (V,E)$ with $\nu$ vertices is \emph{strongly regular}, if there are integers $\lambda$  and $\mu$ such that every two adjacent vertices have $\lambda$ common neighbors and every two non-adjacent vertices have $\mu$ common neighbors. Those parameters are denoted by $(\nu, r, \lambda,  \mu).$   In Theorem \ref{prop:graph} we prove that $G_\cL(\PR)$ is strongly regular.

\begin{thm}\label{prop:graph} Let $G_\cL(\PR)$ be the graph of lines of a projective rectangle $\PR$ of order $(m,n)$. Then:
 \begin{enumerate}[{\rm(a)}]
 
 \item \label{orderG}  $G_\cL(\PR)$ has $\nu = n^2$ vertices.

 \item \label{regularG}  $G_\cL(\PR)$ is $r = (m+1)(n-1)$-regular.
 
 \item \label{stronglyregularG}  $G_\cL(\PR)$ is strongly regular with parameters $\lambda = m(m-1)+n-2 = n+(m+1)(m-2)$ and $\mu = m(m+1)$.  
 
 \item \label{eigenvaluesG}  The eigenvalues of $G_\cL(\PR)$ are $\tau_0=(m+1)(n-1)$ with multiplicity $1$, $\tau_1=n-m-1$ with multiplicity $(m+1)(n-1)$, and $\tau_2=-(m+1)$ with multiplicity $(n-m)(n-1)$.

 \item \label{diameterG} The diameter of $G_\cL(\PR)$ is $2$ if\/ $\PR$ is nontrivial.  $G_\cL(\PR)$ is a complete graph if\/ $\PR$ is trivial.
 
 \item \label{cliqueinG}  Suppose $\PR$ is nontrivial.  The maximal cliques in $G_\cL(\PR)$ are the plane cliques, of size $m^2$, and the point cliques, of size $n$, all of which are distinct sets.  There are $n^2(n-1)/m^2(m-1)$ plane cliques and $(m+1)n$ point cliques.  The maximum clique size is $n$.

Every vertex of $G_\cL(\PR)$ is in exactly $(n-1)/(m-1)$ plane cliques and exactly $m+1$ point cliques.
Moreover, each two adjacent vertices are together in exactly one plane clique and exactly one point clique.

Two plane cliques, and also two point cliques, intersect in at most one vertex.  A plane clique and a point clique are disjoint or intersect in exactly $m$ vertices.

 \item \label{connetedG}  The connectivity of $G_\cL(\PR)$ is $(m+1)(n-1)$.
 \end{enumerate}
\end{thm}

\begin{proof}
Parts \eqref{orderG} and \eqref{regularG} are immediate from observations in Section \ref{sec:pr}.

Let $l$ be an arbitrary ordinary line in $\PR$ and let $p$ be a point in $l$.  Exactly $n-1$ ordinary lines other than $l$ contain the point $p$.  The same number of ordinary lines, other than $l$, contain each point in $l$.  The only line that contains more than one point of $l$ is $l$, so the number of ordinary lines that intersect $l$ is $(m+1)(n-1)$.
Thus, the degree of the vertex $l$ of $G_\cL(\PR)$ is $(m+1)(n-1)$.  This is true for every vertex of $G_\cL(\PR)$, so $G_\cL(\PR)$ is $(m+1)(n-1)$-regular.

We prove Part \eqref{diameterG}.  If\/ $\PR$ is trivial, every ordinary line is adjacent to every other and $G_\cL(\PR)$ is complete.  If\/ $\PR$ is nontrivial then there exist nonintersecting ordinary lines so $G_\cL(\PR)$ is incomplete.  By Part \eqref{regularG} a nontrivial projective rectangle is a strongly regular graph with $\lambda>0$; therefore its diameter is at most 2, and because $G_\cL(\PR)$ is incomplete its diameter is exactly 2.

We prove Part \eqref{cliqueinG}. 
An ordinary line that is adjacent to every ordinary line of a plane $\pi$ must contain more than one point of $\pi$ and therefore is a line of $\pi$.  Therefore, a plane clique is a maximal clique in $G_\cL(\PR)$.

An ordinary line $l$ that is adjacent to every ordinary line on a point $p$ but does not contain $p$ must intersect all the $n$ ordinary lines on $p$ in distinct points.  Therefore, $m=n$ so $\PR$ is a projective plane, which is a trivial projective rectangle.  Hence, for a nontrivial projective rectangle, a point clique is a maximal clique of $G_\cL(\PR)$.

Consider a maximal clique $K$ that is not a point clique.  Let $l_1,l_2 \in K$; then $l_1 \cap l_2$ is a point $p$ and there is a plane $\pi \supset l_1,l_2$.  Since $K$ is not a point clique, there is an ordinary line $l \in K$ adjacent to $l_1$ and $l_2$ that does not contain $p$, which means $l$ intersects $l_1$ and $l_2$ at distinct points $p_1$ and $p_2$, which are points of $\pi$.  Since $l$ is ordinary, $p_1$ and $p_2$ are not contained in a special line; therefore there is a unique ordinary line on $p_1$ and $p_2$ in $\pi$, which must be the unique line $l$ that contains both those points.  That is, any ordinary line in $K$ that is not on $p$ is a line of $\pi$.  Now suppose $l'$ is an ordinary line in $K$ that is not a line of $\pi$.  If $p \notin l'$, then $l'$, like $l$, is a line of $\pi$.  Therefore, $p \in l'$.  Since $l$ and $l'$ are in a clique, they must intersect; let $q$ be their intersection point.  As $p \notin l$, $q \neq p$.  As $l$ is in $\pi$, $q$ is in $\pi$, as is $p$.  Thus, $l'$ is an ordinary line that contains two points of $\pi$, which implies it is a line of $\pi$.  It follows that every maximal clique that is not a point clique must be a plane clique.

The number of plane cliques in $G_\cL(\PR)$ is the number of planes in $\PR$, which is given by Theorem \ref{thm:counting} Part \eqref{prop:counting:pi}.  The number of point cliques in $G_\cL(\PR)$ is the number of ordinary points in $\PR$.

By Theorem \ref{thm:counting} Part \eqref{prop:counting:plonpi}, the sizes of point and plane cliques are $n$ and $m^2$, respectively.  By Theorem \ref{thm:constraint}, since $\PR$ is nontrivial, $n \geq m^2$.  Therefore, the clique number of $G_\cL(\PR)$ is $n$.

The number of plane cliques that contain a vertex $l$ of $G_\cL(\PR)$ is the number of planes in $\PR$ that contain the line $l$, which is given by Theorem \ref{thm:counting} Part \eqref{prop:counting:pionl}.  The number of point cliques that contain the vertex $l$ is the number of points in $l$.

Two adjacent vertices are ordinary lines $l$ and $l'$ that intersect at a point $p$.  The only point clique of $G_\cL(\PR)$ that contains both vertices is the one that consists of all ordinary lines on $p$.  
By Theorem \ref{prop:twolinesintersectingpp} there is only one plane that contains both lines, so only one plane clique contains both vertices $l$ and $l'$ of $G_\cL(\PR)$.

The intersection of two point cliques defined by points $p$ and $p'$ in $\PR$ is empty if $p$ and $p'$ belong to the same special line.  Otherwise, there is a unique ordinary line $l$ that contains both points, so the vertex $l$ is the only one in the intersection of the two point cliques.

Suppose the intersection of two plane cliques contains vertices $l$ and $l'$.  Then two planes contain both the ordinary lines $l$ and $l'$, which intersect at a point because they are coplanar, but this contradicts Theorem \ref{prop:twolinesintersectingpp}.

We prove Part \eqref{stronglyregularG}.  We may assume $\PR$ is nontrivial.  

First we evaluate $\mu$.  Let $l$ and $l'$ be nonadjacent vertices of $G_\cL(\PR)$; that is, they are nonintersecting ordinary lines in $\PR$.  The number of common neighbors equals the number of ordinary lines that intersect both $l$ and $l'$, which is $(m+1)^2$, the number of lines spanned by one point in $l$ and one in $l'$, less $m+1$, the number of them that are special.

Now we evaluate $\lambda$.  Two adjacent vertices $l$ and $l'$ together belong to exactly one plane clique, say $K_\pi$, which contains $m^2-2$ vertices other than $l$ and $l'$, and one point clique defined by the point $p=l\cap l'$, say $K_p$, which contains $n-2$ vertices other than $l$ and $l'$.  Because every common neighbor of $l$ and $l'$ is in a maximal clique that contains both, and plane and point cliques are the only maximal cliques, the vertices in these cliques are the only ones adjacent to both $l$ and $l'$.  By Part \eqref{cliqueinG} $K_\pi \cap K_p$ consists of exactly $m$ vertices, of which $m-2$ are different from $l$ and $l'$. Thus, the number of common neighbors of $l$ and $l'$ equals $[m^2 - 2] + (n-2) - (m-2) = n + m(m-1) - 2$.  This is the value of $\lambda$.

For Part \eqref{eigenvaluesG}, the eigenvalues and multiplicities follow by standard formulas \cite{gr}.

Part \eqref{connetedG} is from the theorem of \cite{b-m} that the connectivity of a strongly regular graph is $r$.

That completes the proof.
\end{proof}

As insurance, we tested the parameters of $G_\cL(\PR)$ against the Krein bounds \cite[Theorem 10.7.1]{gr} for a nontrivial projective rectangle and verified them with the aid of Theorem \ref{thm:constraint}.

\begin{exam}\label{ex:cliqual}
Sometimes a point clique, consisting of several ordinary lines belonging to distinct planes, is a clique of order
$m^2$, the same as a plane clique.  That occurs when $n = m^2$, the least possible for a nontrivial projective rectangle (by Theorem \ref{thm:constraint}).  For instance the set of vertices $\{l_{1}, l_{4},l_{11},l_{14}\}$ of the graph in Figure \ref{figureGPR} gives rise to $K_{4}$.
However, in Figure \ref{figure1} we can see that there is not a plane containing any two of the lines $l_{1}$, $l_{4}$, $l_{11}$, and $l_{14}$.
\end{exam}

\begin{exam}\label{TensorProduct4}  
The graph  $G_\cL(\PR) $ depicted in Figure \ref{figureGPR} is the graph from the projective rectangle $L_2^2$ in Figure \ref{figure1}.
The vertices are the lines
\begin{center}
\begin{tabular}{lll}
$l_{0}=\{A_{1},B_{1},C_{1} \}$, 	& $l_{1}=\{A_{1},B_{g},C_{g} \}$, 	& $l_{2}=\{A_{1},B_{h},C_{h} \}$,\\
$l_{3}=\{A_{1},B_{gh},C_{gh} \}$,  	& $l_{4}=\{A_{g},B_{g},C_{1} \}$, 	&$l_{5}=\{A_{g},B_{1},C_{g} \}$,\\
$l_{6}=\{A_{g},B_{gh},C_{h} \}$, 	&$l_{7}=\{A_{g},B_{h},C_{gh} \}$, 	& $l_{8}=\{A_{h},B_{h},C_{1} \}$, \\
$l_{9}=\{A_{h},B_{gh},C_{g} \}$, 	&$l_{10}=\{A_{h},B_{1},C_{h} \}$, 	&$l_{11}=\{A_{h},B_{g},C_{gh} \}$,\\
$l_{12}=\{A_{gh},B_{gh},C_{1} \}$, 	&$l_{13}=\{A_{gh},B_{h},C_{g} \}$, 	&$l_{14}=\{A_{gh},B_{g},C_{h} \}$,\\
$l_{15}=\{A_{gh},B_{1},C_{gh} \}$.	&&
\end{tabular}
\end{center}
The graph $G_\cL(\PR) $ is an $(m+1)(n-1)=9$-regular graph with $n^2=16$ vertices. It is strongly regular with $\lambda=4$ and $\mu=6$. 
The graph  $\overline{G}_{\PR} $ depicted in Figure \ref{figureGPR} is the complement of $G_\cL(\PR) $. The
parameters for the complement are $(\nu,\nu-k-1,\nu-2-2k+\mu,\nu-2k+\lambda)$. So,  the complement $\overline{G}_{\PR}$ has parameters $(16, 6, 2, 2)$. It is not bipartite. 
The automorphism group of $G_\cL(\PR) $ has order 1152 and a single orbit. The graph has the Hamilton cycle $l_{4},l_{5},l_{6},l_{7},l_{8},l_{9},l_{10}, l_{11},l_{3},l_{2},l_{1},l_{0},l_{15},l_{14},l_{13}, l_{12},l_{4}$.

\begin{figure} [!ht]
\begin{center}
\includegraphics[width=75mm]{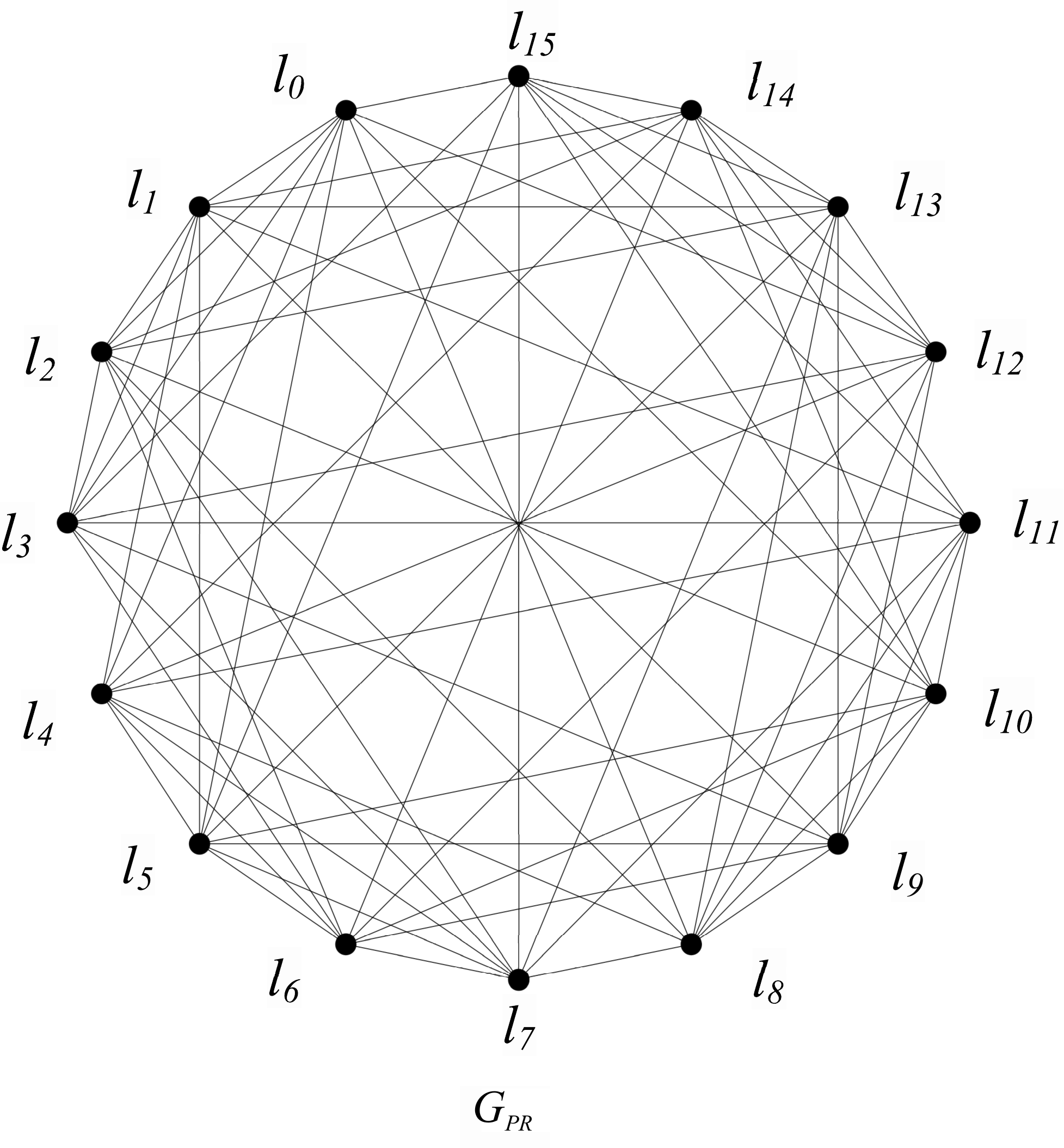} \hspace{.5cm}
\includegraphics[width=75mm] {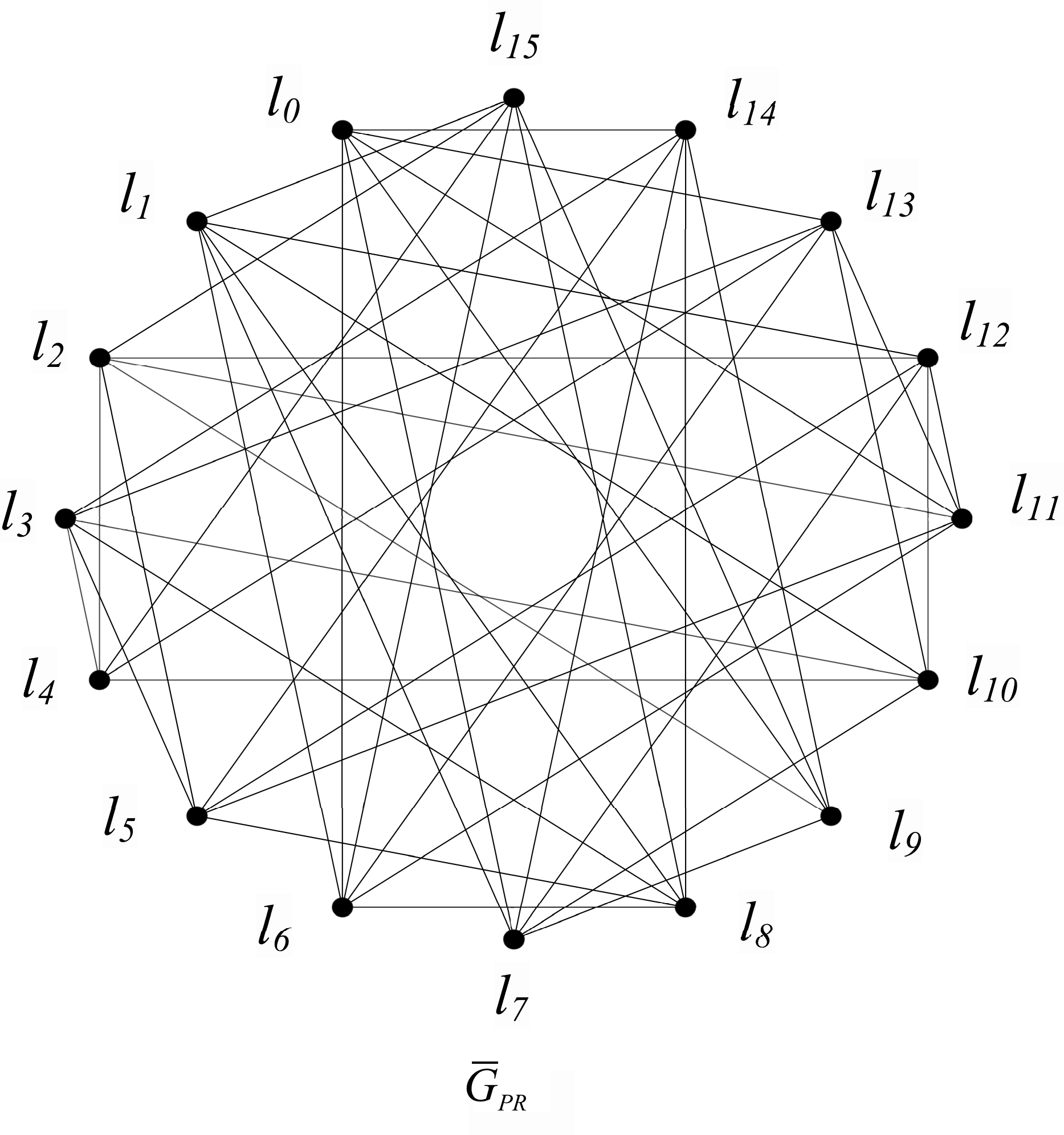}
\caption{The graph of lines $G_\cL(L_2^2)$ and its complement.} \label{figureGPR}
\end{center}
\end{figure}

The graph $G_\cL(\PR)$ is the tensor product $K_{4}\times K_{4}$.
For the proof, let $G_1=\{l_{0},l_{1},l_{4},l_{5} \}$, $G_2=\{l_{2},l_{3},l_{6},l_{7} \}$, $G_3=\{l_{8},l_{9},l_{12},l_{13} \}$, and $G_4=\{l_{10},l_{11},l_{14},l_{15} \}$ observe that the elements of $G_i$ form a complete graph for $i=1,2,3,4$. If $v \in G_i$, then there are exactly two vertices in $u,w \in G_j$ adjacent to $v$ and if $x$ is another vertex of $G_j$, then $x$ is adjacent to exactly one of these two vertices. This completes the proof.
\end{exam}

%============
\sectionpage
\section {The strongly regular graph and a construction of projective rectangles}\label{sec:srg}

The graph of lines, $G_\cL(\PR)$, is strongly regular, but which strongly regular graph is it?  We can identify it as a known type in some examples.  (We use field notation in this section, in particular $\bbF_2$ instead of $\bbZ_2$.)

\begin{exam}[Narrow Projective Rectangles]\label{ex:narrow}
The narrow projective rectangles $\PR = L_2^k$ from Example \ref{ex:L2k} are the projective rectangles with $m=2$ \cite[Section {sec:narrow}]{pr1}.  
According to Theorem \ref{prop:graph} the strongly regular graph $G_\cL(L_2^k)$ has parameters $(\nu,r,\lambda,\mu) = (4^k,3(2^k-1),2^k,6)$.  In small cases the parameters are:
\begin{center}
%\begin{table}[h]
\begin{tabular}{r|rrrr}
	&$\nu$	&$r$	&\ \ $\lambda$ &\ $\mu$	 \\
\hline
$k$	&$4^k$	&$3(2^k-1)$ &$2^k$ &6	\\\hline
2	&16	&9	&4	&6	\\
3	&64	&21	&8	&6	\\
4	&256	&45	&16	&6	\\
5	&1024	&93	&32	&6	\\
6	&4096	&189	&64	&6	\\
\end{tabular}
%\bigskip
%\caption{The parameters of $G_\cL(L_2^k)$.}
%\end{table}
\end{center}
\end{exam}

These parameters agree with those of the strongly regular bilinear forms graphs $H_q(2,k)$ from \cite[Section 9.5A]{bcn} (or see \cite[Section 3.4.1]{bvm}).  This fact led us to a characterization of the graph of lines of a narrow rectangle.  The graph $H_q(2,k)$ has as vertices the $2\times k$ matrices over $\bbF_q$; two matrices are adjacent when their difference has rank 1.\footnote{We thank Andries Brouwer for advice on the bilinear forms graphs and for his tables of parameters \cite{btables}, where $H_q(2,k)$ is named Bilin$_{2\times k}(q)$.}

\begin{thm}\label{srg2}
The graph of lines of $L_2^k$ is isomorphic to the bilinear forms graph $H_2(2,k)$.
\end{thm}

\begin{proof}
The proof requires us to use the natural coordinates of $L_2^k$ (Example \ref{ex:L2k}), which we describe in the following example.

\begin{exam}[Coordinates]\label{ex:L2kcoordinates}
Since $L_2^k$ is the complete lift matroid $\bL_0(\bbF_2^k {\cdot} K_3)$, it comes provided with a coordinate system in the vector space $\bbF_2^k$, which we treat as the additive group $\bbF_{2^k}^+$ of the field $\bbF_{2^k}$.  That lets us treat $L_2^k$ as a subset of the projective plane $\PP(\bbF_{2^k})$.  We use homogeneous coordinates $[x,y,z]$ for points.  A line has an equation $ax+by+cz=0$; we denote the line by a homogeneous triple $\langle a,b,c \rangle$.

The special point is $D: [0,0,1]$.  

The special lines are $\langle a,b,0 \rangle$, where $a,b \in \bbF_2$ (but not both 0); that is, they are $s_0: \langle 1,0,0 \rangle$, $s_\infty: \langle 1,1,0 \rangle$, and $s_1: \langle 0,1,0 \rangle$.  

The ordinary points have coordinates $[x,y,z]$ with $x,y,z \in \bbF_{2^k}$ but $x,y$ not both 0; but they are more restricted because they are in $s_\infty \cup s_0 \cup s_\infty$.  A point in $s_0$ has $a=0$; a point in $s_1$ has $b=0$; a point in $s_\infty$ has $a=b$.

The ordinary lines $l$ are $\langle a,b,1 \rangle$ where $a,b \in \bbF_{2^k}$.  The third line coordinate is nonzero because $D \notin l$; we can standardize the line coordinates to end with 1 because they are homogeneous.

Two ordinary lines $l_1: \langle a_1,b_1,1 \rangle$ and $l_2: \langle a_2,b_2,1 \rangle$ are adjacent if and only if they have a point in common.  They have a common point in $s_0$ if and only if $b_1=b_2$, in $s_1$ if and only if $a_1=a_2$, and in $s_\infty$ if and only if $(a_2-a_1) + (b_2-b_1) = 0$; otherwise they have no common point (in $L_2^k$).  The proof is that the common point is $$[x,y,z] = \big[-(b_2-b_1), a_2-a_1, \begin{vmatrix}a_1&a_2\\b_1&b_2\end{vmatrix}\big].$$ 
Thus, in the graph of lines there are three kinds of adjacency: $s_i$-adjacency for $i = 0,1,\infty$ depending on the location of the common point.  (This gives a factorization of the graph into three $(n-1)$-regular subgraphs.)
\end{exam}

We proceed with the proof.  For each ordinary line $l = \langle a,b,1 \rangle$ we construct a $2\times k$ matrix $M(l)$ over $\bbF_2$.  We assume a fixed $\bbF_2$-basis for the vector space $\bbF_{2^k}^+$ and denote by $B(x)$ the coordinate vector (a row vector) of $x \in \bbF_{2^k}$.  Then $M(l) := \begin{pmatrix} B(a) \\ B(b) \end{pmatrix}.$

The essential question is about the difference $M(l_2)-M(l_1)$.  We compute it:
$$
M(l_2)-M(l_1) = \begin{pmatrix} B(a_2)-B(a_1) \\ B(b_2)-B(b_1) \end{pmatrix} 
= \begin{pmatrix} B(a_2-a_1) \\ B(b_2-b_1) \end{pmatrix} .
$$
This matrix has rank 1 if and only if either a row is zero, or one row is a nonzero scalar multiple of the other, which over $\bbF_2$ means they are equal.  One row is zero if and only if $l_1 \cap l_2$ is a point in $s_0$ or $s_1$.  The rows are equal if and only if $a_2-a_1 = b_2-b_1,$ which can be rewritten as $(a_2-a_1) + (b_2-b_1) = 0$, the condition for a common point in $s_\infty$.  Thus, two ordinary lines are adjacent if and only if their matrices are adjacent in $H_2(2,k)$, which is the theorem.
\end{proof}

More general examples show the same numerical agreement.  The rectangles here are not assumed to satisfy Axiom (A\ref{Axiom:A6}); hence we call them pseudo-projective.  The treatment is like that of $L_2^k$ but is based on projective coordinates rather than the complete lift matroid structure.  We begin with the definition.

\begin{construction}[Subplane Construction {\cite[Section 6]{pr1}}]\label{ex:subplanec}
Theorem \ref{srg2} generalizes to a \emph{subplane construction} that sometimes is a projective rectangle.  Let $q$ be a prime power and $k>1$.  The projective plane $\Pi=\PP(q^k)$ contains a subplane $\pi=\PP(q)$.  Pick a point $D \in \pi$ and let $\cS$ be the set of all lines of $\Pi$ of the form $\overline{Dp}$ for $p \in \pi$; let $\cP = \bigcup \cS$.  Finally, let $\cO$ be the set of all restrictions $l=L\cap\cP$ to $\cP$ of lines $L$ of $\Pi$ that do not contain $D$.  We use the terminiology of projective rectangles; the lines in $\cS$ are special and those in $\cO$ are ordinary.  We shall call the structure $R(q,q^k)=(\cP,\cS\cup\cO)$ a \emph{pseudo-projective rectangle}.  It satisfies all the axioms of a projective rectangle except (possibly) (A\ref{Axiom:A6}).  

When $q$ is a prime number $R(q,q^k)$ does satisfy (A\ref{Axiom:A6}) so it is a projective rectangle; that is a special case of \cite[Theorem 6.3]{pr1}.  
In Theorem \ref{thm:pr} we shall give a different proof that every pseudo-projective rectangle satisfies (A\ref{Axiom:A6}).
\end{construction}

Define two ordinary lines to be adjacent if they have a point in common; this defines the \emph{graph of lines} $G_\cL(R(q,q^k))$.  The numerical results in Theorem \ref{prop:graph} that do not depend on the existence of planes in $R(q,q^k)$ remain valid; that is everything except what involves plane cliques.  
We note that by \cite[Theorem 6.7]{pr1} the existence of a plane clique on every pair of adjacent lines is equivalent to $R(q,q^k)$ being a projective rectangle.

\begin{exam}[Subplane Rectangles]\label{ex:subplaner}
In \cite[Section 6]{pr1} we constructed a pseudo-projective rectangle in $\PP(q^k)$ with parameters $(q,q^k)$ where $q$ is a prime power and $k>1$.  Its graph of lines has strongly regular parameters $(\nu,r,\lambda,\mu) = (q^{2k}, (q+1)(q^k-1), q^k+(q+1)(q-2), q(q+1))$.  When $q=2$ these are the rectangle and graph in the previous example.  For larger values of $q$ we have the data in Table \ref{tb:subplanesrg}.  These parameters agree with those for $H_q(2,k)$.

\begin{center}
\begin{table}[h]
\begin{tabular}{r|rrrr}
	&$\nu$	&$r$	&$\lambda$ &$\mu$	\\
\hline
$k$	&$9^k$	&$4(3^k-1)$ &$3^k+4$ &12	\\\hline
2	&81	&32	&13	&12	\\
3	&729	&104	&31	&12	\\
4	&6561	&320	&85	&12	\\
\end{tabular}
\qquad
\begin{tabular}{r|rrrr}
	&$\nu$	&$r$	&$\lambda$ &$\mu$	\\
\hline
$k$	&$16^k$	&$5(4^k-1)$ &$4^k+10$ &20	\\\hline
2	&256	&75	&26	&20	\\
3	&4096	&315	&74	&20	\\
4	&65636	&1275	&266	&20	\\
\end{tabular}
\bigskip
\caption{The parameters of $G_\cL(\PR)$ for $q=3$ and $4$.}
\label{tb:subplanesrg}
\end{table}
\end{center}
\end{exam}

We use homogeneous coordinates $[x,y,z]$ for points in $\PP(q^k)$.  A line has an equation $ax+by+cz=0$ with $a,b,c \in \PP(q^k)$, not all 0; we denote the line by a homogeneous triple $\langle a,b,c \rangle$.
The special point is $D: [0,0,1]$.  
The special lines are $\langle \alpha,\beta,0 \rangle$, where $\alpha,\beta \in \bbF_q$ (but not both 0); that is, they are $s_ \beta: \langle 1, \beta,0 \rangle$ for $\beta\in \bbF_q$ and $s_\infty: \langle 0,1,0 \rangle$.  
The ordinary points have coordinates $[x,y,z]$ with $x,y,z \in \bbF_{q^k}$ but $x,y$ not both 0; but they are more restricted because they are in $s_\infty \cup \bigcup_{\beta\in \bbF_q} s_\beta $.  Thus, either $y=0$ (the point is in $s_\infty$) or $x = -\beta y$ for some $\beta \in \bbF_q$ (the point is in $s_\beta$).
The ordinary lines $l$ are $\langle a,b,1 \rangle$ where $a,b \in \bbF_{q^k}$.  The third line coordinate is nonzero because $D \notin l$; we can standardize the line coordinates to end with 1 because they are homogeneous.

The common point of ordinary lines $l_1: \langle a_1,b_1,1 \rangle$ and $l_2: \langle a_2,b_2,1 \rangle$ is 
\begin{equation}\label{eq:commonpoint}
[x,y,z] = \big[-(b_2-b_1), a_2-a_1, \begin{vmatrix}a_1&a_2\\b_1&b_2\end{vmatrix}\big],
\end{equation} 
provided this point is in $\cP$.  
To prove this we take the difference $(a_1x+b_1y+z) - (a_2x+b_2y+z) = 0$, whose solution is $x = -\lambda(b_2-b_1)$, $y = \lambda(a_2-a_1)$ for $\lambda \in \bbF_{q^k}$.  Substituting for $x, y$ in $a_1x+b_1y+z = 0$ gives $z = \lambda \begin{vmatrix}a_1&a_2\\b_1&b_2\end{vmatrix}.$  That gives the intersection of $L_1$ and $L_2$, the extensions of $l_1$ and $l_2$ into $\Pi$.  The point \eqref{eq:commonpoint} is in $\cP$ if and only if $[x,y] = [x',y']$ for some $x',y' \in \bbF_q$, as we shall see in detail in the proof of Theorem \ref{srgq}.

\begin{thm}\label{srgq}
The graph of lines of a pseudo-projective rectangle $R(q,q^k)$ in $\PP(q^k)$ is isomorphic to the strongly regular graph $H_q(2,k)$.
\end{thm}

\begin{proof}
Equation \eqref{eq:commonpoint} shows that two ordinary lines $l_1: \langle a_1,b_1,1 \rangle$ and $l_2: \langle a_2,b_2,1 \rangle$ have a common point in $s_0$ if and only if $b_1=b_2$, in $s_\infty$ if and only if $a_1=a_2$, and in $s_ \beta $ for $\beta \in \bbF_{q^k}^\times$ if and only if $a_2-a_1 = -\beta(b_2-b_1)$; otherwise they have no common point in $R$.  
Thus, in the graph of lines there are $q+1$ kinds of adjacency: $s_b$-adjacency for $i = 0,1,\infty$ depending on the location of the common point.  (This gives a factorization of the graph into $q+1$ $(q^k-1)$-regular subgraphs.)

For each ordinary line $l = \langle a,b,1 \rangle$ we construct a $2\times k$ matrix $M(l)$ over $\bbF_q$.  We assume a fixed $\bbF_q$-basis for the vector space $\bbF_{q^k}^+$ and denote by $B(x)$ the coordinate vector (a row vector) of $x \in \bbF_{q^k}$.  Then $M(l) := \begin{pmatrix} B(a) \\ B(b) \end{pmatrix}.$

The essential question is about the difference $M(l_2)-M(l_1)$.  We compute it:
$$
M(l_2)-M(l_1) = \begin{pmatrix} B(a_2)-B(a_1) \\ B(b_2)-B(b_1) \end{pmatrix} 
= \begin{pmatrix} B(a_2-a_1) \\ B(b_2-b_1) \end{pmatrix} .
$$
This matrix has rank 1 if and only if either a row is zero, or one row is a nonzero scalar multiple of the other, which over $\bbF_2$ means they are equal.  One row is zero if and only if $l_1 \cap l_2$ is a point in $s_0$ or $s_\infty$.  The rows are equal if and only if $a_2-a_1$ and $b_2-b_1$ are scalar multiples of each other by some $\lambda \in \bbF_q$ which means they have a common point in $s_{-\lambda}$.  Thus, two ordinary lines are adjacent if and only if their matrices are adjacent in $H_q(2,k)$, which is the theorem.
\end{proof}

\begin{thm}\label{thm:pr}
Every pseudo-projective rectangle is a projective rectangle.
\end{thm}

\begin{proof}
Let $R$ be the pseudo-projective rectangle of order $(q,q^k)$; therefore, $G_\cL(R) \cong H_q(2,k)$.  
According to \cite[page 101]{bvm}, $H_q(2,k)$ has two kinds of maximal clique.  One kind has order $(q^k)^2$ and the other has order $q^2$ (proved in \cite[Lemma 2.2]{huang}).  The former kind is easily seen to be our point clique.  The latter kind has the same order as a plane clique, and the only way such a clique can exist in $G_\cL(R)$ is as a plane clique, because it requires a set of $q^2$ ordinary lines and $q+1$ special lines that are all mutually adjacent, which implies this incidence substructure is a trivial pseudo-projective plane of order $q$ with no two disjoint ordinary lines, hence (as we noted in Section \ref{sec:pr}) a projective plane of order $q$ and a plane of $R$.  
By the remarks just preceding Lemma 2.1 in \cite{huang}, every pair of adjacent vertices is in a clique of the second kind, which we now call a plane clique.  It then follows from \cite[Theorem 6.7]{pr1} that $G_\cL(R)$ is a projective rectangle.
\end{proof}

Theorem \ref{srgq} tells us that $G_\cL$ is not a new strongly regular graph if it is from a finite projective rectangle obtained by the subplane construction.  It does not imply that $G_\cL \cong H_q(2,k)$ for a finite projective rectangle obtained in a different way, if such exist.  We have no guess as to whether such different projective rectangles exist, but we do know that they must have $m$ equal to the order of a projective plane (because of the planes they contain), whatever that may imply.

A highly regular graph like $G_\cL$ suggests there may be a partial geometry (e.g., see \cite{vLW}) hiding in it, whose points are the vertices (the ordinary lines of $\PR$) and whose lines are maximal cliques.  For $G_\cL$ there are two clique types of different sizes, point cliques and plane cliques, suggesting two partial geometries.  For general projective rectangles, taking point cliques as lines gives a partial geometry that is a restatement of the net viewpoint in \cite[Section 9]{pr1}; for the subplane construction in particular, i.e., the graph $H_q(2,k)$, it is mentioned in \cite[Section 3.4.1]{bvm}.  Taking plane cliques for lines does not give a partial geometry but for $H_q(2,k)$ this incidence structure is also mentioned in \cite[Section 3.4.1]{bvm}, where it is called a semi-partial geometry.  

\begin{qn}\label{qn:pg-spg}
Do other finite projective rectangles, if they exist, give new partial geometries and semi-partial geometries?
\end{qn}

%============
\sectionpage
\section {Graph properties}\label{sec:properties}

We present some elementary graph properties of $G_\cL$.

\begin{prop}\label{Coro:graph} Let $G_\cL(\PR)$ be the graph of lines of a projective rectangle of order $(m,n)$.
 \begin{enumerate}[{\rm(a)}]
 \item\label{Coro:graph:a}  $G_\cL(\PR)$ is nonplanar, except when $m=n=2$.
 \item\label{Coro:graph:b}  $G_\cL(\PR)$ has an Eulerian circuit if, and only if, $m$ or $n$ is odd.
 \item\label{Coro:graph:c} If $n \leq 3m+1$, then $G_\cL(\PR)$ is Hamiltonian.
 \end{enumerate}
\end{prop}

\begin{proof} 
Proof of Part (\ref{Coro:graph:a}).  A regular graph is planar only if it has degree at most 5.  To satisfy this, by Theorem \ref{prop:graph} Part \eqref{regularG} we must have $(m+1)(n-1)\leq5$ while $n\geq m \geq 2$.  The only solution is $m=n=2$.  In this case $G_\cL(\PR)$ is the intersection graph of the lines in the Fano plane that do not contain a certain point; this graph is $K_4$, which is planar.

Part (\ref{Coro:graph:b}) is straightforward from Theorem \ref{prop:graph} Parts \eqref{regularG} and \eqref{connetedG}.

Proof of Part (\ref{Coro:graph:c}).  $G_\cL(\PR)$ has $n^2$ vertices.  By
Theorem \ref{prop:graph} Part \eqref{connetedG} we know that  $G_\cL(\PR)$ is $2$-connected. We know that $n>1$ and that $n+1\le 3(m+1)$, therefore
$n+1+1/(n-1) \le 3(m+1)$. This implies that $n^2 \le 3(m+1)(n-1)$.  
(Jackson \cite{jackson} proved that every $2$-connected $k$-regular graph on at most $3k$ vertices is Hamiltonian.) 
Since $G_\cL(\PR)$ is $(m+1)(n-1)$-regular, the conclusion follows.
\end{proof}

Regarding coloring, we have two coloring problems.

\begin{prop}\label{prop:chromaticGlines}
The chromatic number of a nontrivial projective rectangle $\PR$ of order $(m,n)$ satisfies $\chi(G_\cL(\PR)) \geq (n-1)(n-m)$.
\end{prop}

\begin{proof}
Haemers \cite{haem} proved that $\chi(G) \geq \min(\mu_2, 1-\tau_2/\tau_1)$ (in our notation), where $\mu_2$ denotes the multiplicity of the smallest eigenvalue $\tau_2$.  Since 
$$1-\frac{\tau_2}{\tau_1} = \frac{n}{n-m-1} \leq \frac{n}{m^2-m-1} < n$$
by Theorem \ref{thm:constraint} and $m\geq 2$, the result follows from Haemers' lower bound.
\end{proof}

This lower bound is much bigger than the maximum clique size $n$, because (by Theorem \ref{thm:constraint}) $n-m \geq m^2 - m = m(m-1) \geq 2$.  Thus, cliques tell us nothing about the chromatic number.  We propose:

\begin{conj} Equality holds in Proposition \ref{prop:chromaticGlines}.
\end{conj}

The chromatic index of an $r$-regular graph is $r$ or $r+1$ by Vizing's Theorem.  It cannot be $r$ if the graph has odd order $\nu$.  Ferber and Jain give a sufficient condition for the chromatic index to equal the degree, from which we derive an asymptotically valid conclusion.

\begin{prop}\label{prop:chromaticindexGlines}
Let $\PR$ be a nontrivial projective rectangle.  If $n$ is odd, the chromatic index satisfies $\chi'(G_\cL(\PR)) = r+1 = (m+1)(n-1)+1$.
If $n$ is even and sufficiently large and $m+1 \geq \sqrt[9]{n-1}$, then $\chi'(G_\cL(\PR)) = r = (m+1)(n-1)$.
\end{prop}

\begin{proof}
If $n$ is odd, $\nu = n^2$ is odd, so the value of $\chi'$ follows from Vizing's Theorem.  

If $n$ is even, $\nu$ is even.  Ferber and Jain \cite{ferb} (as quoted in \cite[Theorem 1.1]{cioaba}) proved that $\chi'(G) = r$ if $G$ is an $r$-regular graph with $\nu$ vertices, where $\nu$ and $r$ are sufficiently large, and provided that (in our notation) $\max(\tau_1,-\tau_2) < r^{0.9}$.  In $G_\cL(\PR)$, since $m\geq 2$, if $n$ is sufficiently large then $\nu$ and $r$ will be sufficiently large.  Also, $\max(\tau_1,-\tau_2) = \tau_1 = n-m-1 > m+1 = -\tau_2$, so we need to have $n-m-1 < (m+1)^{0.9}(n-1)^{0.9}$.  Since $n-m-1 \leq n-3$, it is sufficient to have $n-1 \leq (m+1)^{0.9}(n-1)^{0.9}$, which simplifies to $n-1 \leq (m+1)^9$.
\end{proof}

This result misses many projective rectangles.  For instance, if $\PR$ is obtained by the subplane construction in the projective plane $\mathbb{PP}(\bbF_{2^k})$ using a subplane $\mathbb{PP}(\bbF_{2^j})$ where $j<k$, then $n=2^k$ and $m=2^j$.  The 9-th root inequality asks that $(2^j+1)^9 \geq 2^k-1$ to apply Proposition \ref{prop:chromaticindexGlines}.  A sufficient condition is that $j \geq k/9$; and $k$ must be sufficiently large.  
Thus, we ask:

\begin{qn}
Is $\chi'(G_\cL(\PR)) = r = (m+1)(n-1)$ for every nontrivial projective rectangle with even $n$?
\end{qn}

%============
\sectionpage
\section {One and one-half partial geometries}\label{sec:partialgeometry}

The high regularity of the graph of lines of a finite projective rectangle $\PR$ suggests there might be a partial geometry hiding in it.  
A partial geometry pg$(k,r,t)$ is a system of Points and Lines in which each Line has $k$ Points, each Point belongs to $r$ Lines, and for each Point-Line pair $(P_0,L_0)$ with $P_0 \notin L_0$ there are $t$ Lines on $P_0$ that intersect $L_0$.  (Note the capitalization of Points and Lines in the partial geometry to distinguish them from points and lines in $\PR$.)

An obvious way to build a partial geometry from a graph is to take its vertices as points and some of its maximal cliques as lines.  Having two sizes of maximal clique in $G_L$ suggests two possible partial geometries.  

\begin{exam}[Using point cliques]
This partial geometry is mentioned in \cite[Section 3.4.1]{bvm} for the graphs $H_q(2,k)$.  Taking a Point to be an ordinary line of $\PR$ and a Line to be the set of lines in a point clique does yield a partial geometry.  Consider a Point-Line pair $(P_0,L_0)$ with $P_0 \notin L_0$.  The number $t(P_0,L_0)$ of Lines $L$ on $P_0$ that intersect $L_0$ should be a constant.  

Let $L_0$ be the point clique of all ordinary lines on an ordinary point $p_0$ in special line $s_0$ and $P_0$ any line $l_0$ not containing $p_0$.  For every point $p$ on $l_0$ that is not in $s_0$, there is one line $\overline{p_0p}$ in $L_0$; therefore, $t(P_0,L_0) = m$ and we have a partial geometry pg$(m+1,n,m)$ with $n^2$ Points, $mn$ Lines, $n$ Points per Line, and $m+1$ Lines on each Point.  

This partial geometry is essentially a dualization of $\PR$ and as such is a restatement of the net viewpoint in \cite[Section 9]{pr1}.
\end{exam}

\begin{exam}[Using plane cliques]
Taking a Point to be an ordinary line of $\PR$ and a Line to be the set of lines in a plane clique does not yield a partial geometry.  Consider a Point-Line pair $(P_0,L_0)$ with $P_0 \notin L_0$.  There are two kinds of such Point-Line pairs $(P_0,L_0)$.  The Point $P_0=l_0$, which an ordinary line of $\PR$, and the Lines $L_0 = \pi_0$, a plane, may be disjoint point sets in $\PR$, or they may have one common point.  In the former case $t(P_0,L_0) = 0$ and in the latter case $t(P_0,L_0) = m$, because the Line $L$ must be a plane that contains both $l_0$ and a line in $\pi_0$.  (This incidence structure from $H_q(2,k)$ is called a semi-partial geometry in \cite[Section 3.4.1]{bvm}.)
\end{exam}

%============
\sectionpage

\end{document}